\documentclass[11pt, a4paper]{article}
\usepackage[ansinew]{inputenc}
\usepackage[english]{babel}
\usepackage{comment}
\usepackage{amsmath}
\usepackage{amssymb}
\usepackage{amsthm}

\usepackage{graphicx}
\usepackage{epstopdf}

\usepackage[backend=bibtex,style=alphabetic,backref=false,doi=false,
isbn=false,url=false]{biblatex}
\bibliography{neighborhood.bib}

\newtheorem{theorem}{Theorem}

\newtheorem{corollary}[theorem]{Corollary}

\newtheorem{example}[theorem]{Example}

\title{Graph Operations and Neighborhood Polynomials}
\author{Maryam Alipour\footnote{This research has been funded by the European Social Fund (ESF).} and Peter Tittmann\\University of Applied Sciences Mittweida}

\begin{document}
	
	\maketitle
	
	\begin{abstract}
		The neighborhood polynomial of graph $G$ is the generating function for 
		the number of vertex subsets of $G$ of which the vertices have a common 
		neighbor in $G$. 
		In this paper, we investigate the behavior of this polynomial under several 
		graph operations. Specifically, we provide an explicit formula for the 
		neighborhood polynomial of the graph obtained from a given graph $G$ by 
		vertex attachment. We use this result to propose a recursive algorithm
		for the calculation of the neighborhood polynomial. Finally, we prove
		that the neighborhood polynomial can be found in polynomial-time in the 
		class of $k$-degenerate graphs.		
	\end{abstract}

\section{Introduction}\label{sect:intro}

All graphs considered in this paper are simple, finite, and undirected. Let $G=(V,E)$ be a graph where $V$ is its vertex set and $E$ 
its edge set, suppose $v\in V$ be a vertex of $G$. The \emph{open neighborhood} 
of $v$, denoted by $N(v)$, is the set of all vertices that are adjacent to $v$,
\[ 
N(v)=\{u\mid (u,v)\in E\}.
\]

The \emph{neighborhood complex} of graph $G$, denoted by $\mathcal{N}(G)$, has been
introduced in \cite{Lovasz1978}. It is the family of all those vertex subsets 
that have a common neighbor in $G$. In other 
words, $\mathcal{N}(G)$ is the family of all subsets of open neighborhoods of 
vertices of graph $G$,
\[ 
\mathcal{N}(G) = \{A \mid A\subseteq V , \exists v\in V : A\subseteq N(v) \}.
\]

In the following, we list some properties of the neighborhood complex $\mathcal{N}(G)$. The proofs can be found in \cite{Brown2008}:
\begin{itemize}
	\item If $G$ does not contain any isolated vertices, then for any vertex 
	$v\in V$, we have $\{v\}\in \mathcal{N}(G)$,
	\item If $A\in \mathcal{N}(G)$ and $B\subseteq A$, then $B\in \mathcal{N}(G)$,
	\item Let $G'$ be the graph obtained from $G$ by adding some isolated vertices, 
	then $\mathcal{N}(G') = \mathcal{N}(G)$. 
\end{itemize}

The \textit{neighborhood polynomial} of graph $G$, denoted by $N(G,x)$, is the ordinary generating function for the 
neighborhood complex of $G$. It has been introduced in \cite{Brown2008} and is defined as follows:
\begin{equation}
N(G,x) = \sum_{U\in \mathcal{N}(G)} x^{|U|}
\label{eq:1}
\end{equation}

Suppose $|V|=n$, and let $n_k(G)=|\{A \mid A\in \mathcal{N}(G),|A|=k\}|$, then we 
can rephrase the Equation~(\ref{eq:1}) as follows:
\[ 
N(G,x)=\sum_{k=0}^{n} n_{k}(G)x^k.
\]

The neighborhood polynomial of a graph is of special interest as it has a close relation to the domination
polynomial of a graph. A \emph{dominating set} of a graph $G=(V,E)$ is a vertex set $W\subseteq V$ such that
the closed neighborhood of $W$ is equal to V, where the \emph{closed neighborhood} is defined by
\[
N[W] = \bigcup_{w\in W}N(w) \cup W.
\]
We denote by $\mathcal{D}(G)$ the family of all dominating sets of a graph $G$.
The \emph{domination polynomial} of a graph, introduced in \cite{Arocha2000}, is
\[
D(G,x) = \sum_{W\in\mathcal{D}(G)}x^{|W|}.
\]
For further properties of the domination polynomial, see \cite{Akbari2010,Dod2016,Dohmen2012,Kotek2012,Kotek2014}.
It has been observed in \cite{Brouwer2009} that a vertex set $W$ of $G$ belongs to the neighborhood complex
of $G$ if and only if $W$ is non-dominating in the complement $(\bar {G})$ of $G$, which implies
\[
D(G,x) + N((\bar {G}),x) = (1+x)^{|V|}.
\]
A proof of this relation is given in \cite{Heinrich2017}.

In Section \ref{sect:operations} of this paper, we investigate the effect of several 
graph operations on the neighborhood polynomial of a graph. The vertex (or edge) addition plays an important role in this context.

In Section \ref{sect:degenerate}, we present a recursion formula for the neighborhood 
polynomial of a graph based of deleting vertices. Finally, we apply this recursion 
to several graph classes and prove that the neighborhood polynomial of \emph{planar} 
graphs can be computed efficiently.

\section{Graph Operations}\label{sect:operations}

Having defined a graph polynomial, one of the first natural problems is its 
calculation. Often local graph operations, like edge or vertex deletions, prove
useful. In addition, global operations, like complementation of forming the line
graph, might be beneficial. Finally, graph products, for instance the disjoint union,
the join, or the Cartesian product, can be employed to simplify graph polynomial
calculations. 

The  \emph{disjoint union} of two graphs $G_{1}=(V_{1},E_{1})$ and $G_{2}=(V_{2},E_{2})$ with disjoint vertex sets $V_1$ and $V_2$, denoted by $G_1\cup G_2$ is a graph with the vertex set $V_{1}\cup V_2$, and the edge set $E_{1}\cup E_2$.

\begin{theorem}[\cite{Brown2008}]
	Let $G_1$ and $G_2$ be simple undirected graphs. The disjoint union $G_1\cup G_2$  
	satisfies
	\begin{equation}
	N(G_1\cup G_{2},x) = N(G_{1},x) + N(G_{2},x) - 1.
	\end{equation}
\end{theorem}

The \emph{join} of two graphs $G_{1}=(V_{1},E_1)$ and $G_{2}=(V_{2},E_2)$ with 
disjoint vertex sets $V_1$ and $V_2$, denoted by $G_{1}+G_{2}$, is the disjoint 
union of $G_1$ and $G_2$, together with all those edges that join vertices in $V_1$ 
to vertices in $V_2$.

\begin{theorem}[\cite{Brown2008}]
	Let $G_{1} = (V_{1},E_{1})$ and $G_{2} = (V_{2},E_{2})$ be simple undirected 
	graphs. The neighborhood polynomial of the join $G_{1} + G_{2}$ of these two 
	graphs satisfies
	\begin{align*}
		N(G_{1}+G_{2},x) &= (1+x)^{|V_{1}|}N(G_{2},x) + (1+x)^{|V_{2}|}N(G_{1},x)\\
		&- N(G_{1},x)N(G_{2},x)
	\end{align*}
\end{theorem}

The \emph{Cartesian product} of graphs $G_{1}=(V_{1},E_1)$ and $G_{2}=(V_{2},E_2)$ 
with disjoint vertex sets $V_1$ and $V_2$, denoted by $G_{1} \square G_2$, is a 
graph with vertex set $V_{1}\times V_{2}=\{(u,v) \mid u\in V_{1}, v\in V_{2}\}$, 
where the vertices $x=(x_{1},x_{2})$ and $y=(y_{1},y_{2})$ are adjacent in 
$G_{1} \square G_2$, if and only if $[x_{1}=y_{1}\text{ and }\{x_{2},y_{2}\}\in E_{2}]$ 
or $[x_{2}=y_{2}\text{ and } \{x_{1},y_{1}\}\in E_{1}]$.

\begin{theorem}[\cite{Brown2008}]\label{theo:cartesian}
	Let $G_{1}=(V_{1},E_{1})$ and $G_{2}=(V_{2},E_{2})$ be simple undirected 
	graphs, then
	\begin{align*}
		N(G_{1}\square G_{2},x) &= 1 + |V_{1}|(N(G_{2},x)-1) 
			+ |V_{2}|(N(G_{1},X)-1) \\
			&+ \sum _{(u,v)\in V(G_{1}\square G_2)} ((1+x)^{|N_{G_{1}}(u)|}-1)((1+x)^{|N_{G_{2}}(v)|}-1)\\
			&- |V_1||V_2|x - 2|E_1||E_2|x^2.
	\end{align*}
\end{theorem}

\subsection{Cut Vertices}

In this section, we consider connected graphs with cut vertices and prove that there 
exists an interesting relation between the neighborhood polynomial of a graph $G$ 
with the neighborhood polynomials of its split components. 

Let $G=(V,E)$ be a simple undirected graph and let $v\in V$. By $G-v$ we denote the 
subgraph of $G$ induced by $V\setminus \{v\}$.
A vertex $v$ in a connected graph $G$ is called a \emph{cut vertex} (or an 
\emph{articulation}) of $G$ if $G-v$ is disconnected. More generally, a vertex 
$v$ is a cut vertex of a graph $G$ if $v$ is a cut vertex of a component of $G$.

Let $v$ be a cut vertex of the graph $G=(V,E)$ and suppose $G-v$ has two components. Then we can find two subgraphs
$G_{1}=(V_{1},E_{1})$ and $G_{2}=(V_{2},E_{2})$ of $G$ such that
\begin{equation*}
	V_{1}\cup V_{2}=V,\qquad V_{1}\cap V_{2}=\{v\}, \qquad E_{1}\cup E_{2}=E, 
	\qquad E_{1}\cap E_{2}=\emptyset.
\end{equation*}
We call the graphs $G_1$ and $G_2$ the \emph{split components} of $G$.

\begin{theorem}\label{cut vertex}
Let $G=(V,E)$ be a simple connected graph where $v\in V$ is a 
	cut vertex of $G$ such that $G-v$ has two components. Let $G_{1}=(V_{1},E_{1})$ 
	and $G_{2}=(V_{2},E_{2})$ be the split components of $G$. Then
	the neighborhood polynomial of $G$ can be computed by
	\begin{align*}
		N(G,x) &= N(G_{1},x) + N(G_{2},x) - (1+x) \\
		&+ ((1+x)^{|N_{G_{1}}(v)|}-1)((1+x)^{|N_{G_{2}}(v)|}-1).  
	\end{align*}
\end{theorem}
\begin{proof}
	Let $X$ be a vertex subset of $G$ with $X\in \mathcal{N}(G)$, which implies 
	that the vertices of $X$ have a common neighbor in $G$. Then we can distinguish 
	the following cases.
	\begin{enumerate}
		\item[(a)] Assume $X\subseteq V_{1}$ or $X\subseteq V_{2}$. Then all 
		possibilities for the selection of $X$ are generated by the polynomial 
		$N(G_{1},x) + N(G_{2},x) - (1+x)$, in which 
		we prevent the double-counting of the empty set and the vertex $v$ by 
		subtracting $1+x$ .
		\item[(b)]  Now assume that $X$ contains at least one vertex from each of 
		$V_{1} \setminus \{v\}$ and $V_{2} \setminus \{v\}$. In this case, we need 
		to count all subsets of the open neighborhood of $v$ in $G$ which include 
		at least one vertex from each of the open neighborhoods of $v$ in $G_{1}$ 
		and $G_{2}$, which is performed by the generating function 
		$((1+x)^{|N_{G_{1}}(v_{1})|}-1)((1+x)^{|N_{G_{2}}(v_{2})|}-1)$.
	\end{enumerate}
\end{proof}

Let $G=(V,E)$ be a simple undirected connected graph. A \emph{vertex cut} (or a \emph{separator}) is a set 
of vertices of $G$ which, if removed together with any incident edges, the remaining graph is disconnected. 
More generally, a vertex subset $W$ is a vertex cut if $W$ is a vertex cut of a component of $G$.\\

Let $W\subseteq V$ be a vertex cut of $G$ which is an independent set in $G$ such that $G-W$ has two components, 
where $G-W$ denotes the subgraph of $G$ induced by $V\setminus W$. We can find two subgraphs 
$G_{1}=(V_{1},E_{1})$ and $G_{2}=(V_{2},E_{2})$ of $G$ such that
\begin{equation*}
	V_{1}\cup V_{2}=V,\qquad V_{1}\cap V_{2}=W, \qquad E_{1}\cup E_{2}=E, 
	\qquad E_{1}\cap E_{2}=\emptyset.
\end{equation*}
We call the graphs $G_{1}$ and $G_{2}$ the \emph{split components} of $G$.
In the following theorem we prove that similar to theorem \ref{cut vertex}, there is a relation between the 
neighborhood polynomial of a graph containing a vertex cut and the neighborhood polynomials of its split 
components.
\begin{theorem}
 Let $G=(V,E)$ be a simple connected graph, $W\subseteq V$ a vertex cut with $|W|=k$ that is also an 
 independent set in $G$ such that $G-W$ has two components. Let $G_{1}=(V_{1},E_{1})$ and $G_{2}=(V_{2},E_{2})$ 
 be the split components 
 of $G$. Then
    \begin{align*}
        N(G,x) &= N(G_{1},x) + N(G_{2},x)  - \sum _{U\subseteq W} A_{U} \\
        &+ \sum _{\substack{U\subseteq W \\ U\neq \emptyset}} (-1)^{|U|+1} ((1+x)^{|\bigcap _{u\in U}N_{G_{1}}(u)|}-1)((1+x)^{|\bigcap _{u\in U}N_{G_{2}}(u)|}-1)
    \end{align*}
where
\begin{equation*}
		A_{U} = \left\lbrace
			\begin{array}{l}
				x^{|U|}\hspace{2mm} \text{ if } \bigcap_{u\in U}N_{G_{1}}(u) \neq \emptyset \hspace{2mm}\text{and}\hspace{2mm} \bigcap_{u\in U}N_{G_{2}}(u) \neq \emptyset,\\
				0 \text{ otherwise.}		
			\end{array}
		\right.
	\end{equation*}
\end{theorem}
\begin{proof}
    Let $X$ be a vertex subset of $G$ with $X\in \mathcal{N}(G)$, which implies 
	that the vertices of $X$ have a common neighbor in $G$. Then we can distinguish 
	the following cases.
	\begin{enumerate}
		\item[(a)] If $X\subseteq V_{1}$ or $X\subseteq V_{1}$, all possibilities to select $X$ are generated 
		by the polynomial 
		\[
		N(G_{1},x) + N(G_{2},x) - \sum _{U\subseteq W} A_{U},
		\]
		in which we avoid any subset of the set $W$, having common neighbors in 
		both split components $G_{1}$ and $G_{2}$, to be counted more than once by subtracting 
		$\sum _{U\subseteq W} A_{U}$.
		\item[(b)] Now assume that $X$ contains at least one vertex from each of 
		$V_{1}\setminus W$ and $V_{2}\setminus W$. In this case, we generate all subsets of the open neighborhoods of vertices of $W$ by the generating function 
		\[
		\sum _{\substack{U\subseteq W \\ U\neq \emptyset}} (-1)^{|U|+1} ((1+x)^{|\bigcap _{u\in U}N_{G_{1}}(u)|}-1)((1+x)^{|\bigcap _{u\in U}N_{G_{2}}(u)|}-1)
		\]
		where by applying the principle of inclusion-exclusion we prevent any double counting. 
	\end{enumerate}	
\end{proof}

\subsection{Matching Edge Cuts}

\begin{figure}[ht]
	\centering
	\includegraphics[scale=0.9]{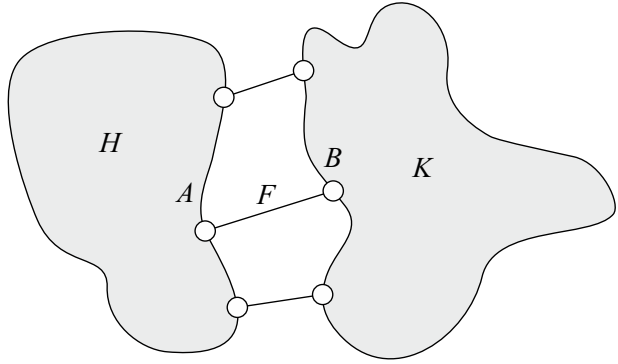}
	\caption{A graph with a matching cut}
	\label{fig:cutmatch}
\end{figure}
\begin{theorem}\label{theo:cutmatch}
	Let $F=\{\{a,b\} \mid a\in A, b\in B \}$ be a minimal cut of a connected graph 
	$G=(V,E)$ such that $F$ is a matching of $G$. The components of $G-F$ are denoted
	by $H$ and $K$. Assume that $A\subseteq V(H)$ and $B\subseteq V(K)$. We define
	$H'=H\cup (A\cup B,F)$ and $K'=K\cup (A\cup B,F)$. Then
	\[ 
	N(G,x) = N(H',x) + N(K',x) -2|F|x -1. 
	\]
\end{theorem}
Figure \ref{fig:cutmatch} illustrates the situation of Theorem \ref{theo:cutmatch}.
\begin{proof}
	The open neighborhoods of all vertices of $V(H)$ are the same in $G$ and in $H'$.
	We also have $N_G(v)=N_{K'}(v)$ for any $v\in V(K)$. This implies
	\[ 
	\mathcal{N}(G) = \mathcal{N}(H') \cup \mathcal{N}(K').
	\]
	Each singleton $\{v\}$ for $v\in A\cup B$ is contained in $\mathcal{N}(H')$ and in
	$\mathcal{N}(K')$, which yields $\mathcal{N}(H')\cap \mathcal{N}(K')=
	\{\{v\}\mid v\in A\cup B \}\cup\{\emptyset \}$.
\end{proof}

\subsection{Edge Addition}

Let $G=(V,E)$ be a graph, and $u,v\in V$ such that $\{u,v\}\notin E$. 
Let $G+uv = (V,E\cup \{\{u,v\} \})$ be the graph obtained from $G$ by insertion 
of the edge $\{u,v\}$. In this case, the neighborhoods of the two vertices 
$u,v$ have been changed. If there is a path of length 3 between the vertices 
$u$ and $v$ in $G$, after the edge insertion, a 4-cycle is being formed which 
causes some difficulties in counting  neighborhood sets. In 
Theorem \ref{theo:edge-insertion}, we suppose that there is no path of length 3 
between the vertices $u$ and $v$ to prevent any over-counting and relate 
the neighborhood polynomial of the new graph after addition the edge $\{u,v\}$ to 
the neighborhood polynomial of the original graph. In Theorem 
\ref{theo:edge-insertion-general}, we investigate the general case.

\begin{theorem}\label{theo:edge-insertion}
	Let $G=(V,E)$ be a graph, $u,v\in V$, and $\{u,v\}\notin E$.
	Suppose there is no path of length 3 between $u$ and $v$ in $G$. Let 
	$G^{\prime}=(V,E\cup \{\{u,v\}\})$. Then
	\begin{equation*}
		N(G^{\prime},x) = N(G,x) + x((1+x)^{|N_{G}(u)|}-1) + x((1+x)^{|N_{G}(v)|}-1).
	\end{equation*}
\end{theorem}
\begin{proof}
	Observe that any vertex subset with a common neighbor in $G$ forms a 
	vertex subset with a common neighbor in $G^{\prime}$. Such vertex subsets 
	are generated by $N(G,x)$. The addition of the edge $\{u,v\}$ to $G$ 
	results in the fact that any non-empty subset of the open neighborhood of 
	$u$ together with $v$ forms a vertex subset with $u$ as the 
	common neighbor in $G^{\prime}$ and the same arguments applies to the 
	open neighborhood of $v$. Such subsets are generated by 
	$x((1+x)^{|N_{G}(u)|}-1) + x((1+x)^{|N_{G}(v)|}-1)$.
	
	Now, assume there is a vertex subset $W$ with a common neighbor in $G^{\prime}$ 
	and suppose it is counted once in 
	$x((1+x)^{|N_{G}(u)|}-1) + x((1+x)^{|N_{G}(v)|}-1)$ and once in 
	$N(G,x)$, and let $y\in W$. This implies that $W$, which is a subset of the open 
	neighborhood of one of the vertices $u$(or $v$) together with 
	$v$ (or $u$) have a common neighbor in $G$. Let $x$ be that common neighbor. The existence of such set results in the existence of a path of 
	length 3 between $u_{1}$ and $u_{2}$ through vertices $x$ and $y$ which 
	contradicts the assumption that there is no path of length 3 between those two 
	vertices in $G$. So there is no vertex subset with a common neighbor in 
	$G^{\prime}$ that is counted twice and this completes the proof.
\end{proof}

\begin{theorem}\label{theo:edge-insertion-general}
	Let $G=(V,E)$ be a graph, $u_{1},u_{2}\in V$, and 
	$(u_{1},u_{2})\notin E$. Suppose $u_{1}$ and $u_{2}$ are not isolated vertices 
	and let $G^{\prime}=(V,E\cup \{\{u_{1},u_{2}\}\})$. Then
	\begin{equation*}
		N(G^{\prime},x) = N(G,x) + x\hspace{-1em}\sum_{\substack{\emptyset \neq 
		U_{1}\subseteq N_{G}(u_{1})\\ 
			U_{1}\cap N_{G}(u_{2})=\emptyset}}\hspace{-1em} x^{|U_{1}|} 
		+ x \hspace{-1em}\sum_{\substack{\emptyset \neq U_{2}\subseteq N_{G}(u_{2})\\ 
		U_{2}\cap N_{G}(u_{1})=\emptyset}}\hspace{-1em} x^{|U_{2}|}.
	\end{equation*}
\end{theorem}
\begin{proof}
	If $X\in \mathcal{N}(G)$ then $X\in \mathcal{N}(G^{\prime})$ and all such 
	vertex subsets are generated by $N(G,x)$. After insertion of the edge 
	$\{u_{1},u_{2}\}$, the vertices in any subset of the open neighborhood of 
	$u_{1}$ together with $u_{2}$ have a common neighbor in $G^{\prime}$ 
	(which is the vertex $u_{1}$), and analogously the vertices in any subset 
	of the open neighborhood of $u_{2}$ together with $u_{1}$ share a neighbor in 
	$G^{\prime}$ (which is the vertex $u_{2}$). Any of such vertex sets adds a  
	term like $x x^{|U|}$ to the neighborhood polynomial of $G^{\prime}$, 
	where $x$ represents $u_{1}$ or $u_{2}$, and $x^{|U|}$ stands for $U$ as a 
	subset of $N_{G}(u_{2})$ or $N_{G}(u_{1})$, respectively. In order to 
	prevent counting a vertex set of this form that already has a common neighbor 
	in $G$ (and therefore being counted in $N(G,x)$) we consider only those subsets 
	of $N_{G}(u_{2})$ or $N_{G}(u_{1})$ which do not have a common neighbor with 
	$u_{1}$ or $u_{2}$ in $G$, respectively. All such vertex subsets are generated 
	by the term 
	\[ 
		x\hspace{-1em}\sum_{\substack{\emptyset \neq U_{1}\subseteq N_{G}(u_{1})\\
				U_{1}\cap N_{G}(u_{2})=\emptyset}}\hspace{-1em} x^{|U_{1}|} 
		+ x \hspace{-1em}\sum_{\substack{\emptyset \neq U_{2}\subseteq N_{G}(u_{2})\\
				 U_{2}\cap N_{G}(u_{1})=\emptyset}}\hspace{-1em} x^{|U_{2}|},
	\]
	which completes the proof.
\end{proof}

\subsection{Vertex Attachment}

Suppose $G=(V,E)$ is a simple graph and let $U\subseteq V$. Let 
\[ 
	G_{U\rhd v} = (V\cup \{v\},E\cup \{\{u,v\}\mid u\in U \})
\]
be the graph obtained 
from $G$ by adding a new vertex $v$ to $V$ and attaching $v$ to all vertices of
$U$, so the degree of $v$ in $G_{U\rhd v}$ is $|U|$; in this case the set 
$U$ is called the \emph{vertex attachment set}. In case of $U=\emptyset$, the graph 
$G_{U\rhd v}$ is the disjoint union of $G$ and the single vertex $v$, which implies
$\mathcal{N}(G_{U\rhd v})=\mathcal{N}(G)$ and therefore $N(G_{U\rhd v},x)=N(G,x)$. 
In the following theorem, we investigate the neighborhood polynomial of $G_{U\rhd v}$ 
in case of $U\neq \emptyset$.

\begin{theorem}\label{theo:v-attach}
	Let $G=(V,E)$ be a graph and $U\subseteq V$ be a non-empty set. For each vertex
	subset $W\subseteq U$, we define a monomial $A_W$ by
	\begin{equation*}
		A_W = \left\lbrace
			\begin{array}{l}
				x^{|W|} \text{ if } \bigcap_{w\in W}N_{G}(w)=\emptyset,\\
				0 \text{ otherwise.}		
			\end{array}
		\right.
	\end{equation*}
	Then	
	\begin{equation*}
		N(G_{U\rhd v},x) = N(G,x)+ \sum_{\substack{W\subseteq U \\ 
			W\neq\emptyset}} (-1)^{|W|+1} 
		x(1+x)^{|\cap_{w\in W}N_{G}(w)|}
		 	+ \sum_{\substack{W\subseteq U\\W\neq\emptyset}} 
		 	 A_W.
	\end{equation*}
\end{theorem}
\begin{proof}
	Assume $X\in \mathcal{N}(G_{U\rhd v})$, which implies that the vertices of 
	$X$ have a common neighbor in $G_{U\rhd v}$. Then, we can distinguish the 
	following cases:
	\begin{enumerate}
		\item[(a)] Assume $v\notin X$ and $v$ is not a common neighbor of the 
		vertices in $X$. In this case $X\in \mathcal{N}(G)$ and all possibilities 
		for the choice of $X$ are counted in $N(G,x)$.
		\item[(b)] Assume $v\notin X$ but $v$ is a common neighbor of the vertices 
		in $X$. Since the only neighbors of $v$ are the vertices in $U$, then we 
		only need to count those subsets of $U$ which do not have a common neighbor 
		in $G$. We do so by
		\[ 
			\sum_{\substack{W\subseteq U\\W\neq\emptyset}} A_W,
		\]
		 where 
		\begin{equation*}
		A_W = \left\lbrace
		\begin{array}{l}
		x^{|W|} \text{ if } \bigcap_{w\in W}N_{G}(w)=\emptyset,\\
		0 \text{ otherwise.}		
		\end{array}
		\right.
		\end{equation*}
		\item[(c)] Finally assume $v\in X$. In this case, we need to count all vertex 
		subsets of the open neighborhoods of vertices $u\in U$ where each one of 
		those subsets together with $v$ forms a vertex subset with $u$ as their 
		common neighbor in $G_{U\rhd v}$. To exclude double counting subsets of 
		intersections of the neighborhoods of $u$'s, we use the principle of 
		inclusion-exclusion. Such subsets are generated by
		\[ 
			\sum_{\substack{W\subseteq U \\ W\neq\emptyset}} (-1)^{|W|+1} 
			x(1+x)^{|\cap_{w\in W}N_{G}(w)|},
		\] 
		where the factor $x$ accounts for the vertex $v$ itself. 
	\end{enumerate}
	Finally,the arguments in (a), (b), and (c) together prove the theorem.
	\end{proof}

In the following corollary we rephrase Theorem \ref{theo:v-attach} in a way that 
the neighborhood polynomial of a graph resulting from vertex removal (instead of
vertex attachment) is investigated.

\begin{corollary}\label{coro:vertex-removal}
	Let $G=(V,E)$ be a graph and $v\in V$. Then we have
	\begin{equation*}		
		N(G,x) = N(G-v,x) +
		\hspace{-0.5em}\sum_ {\substack{U\subseteq N_{G}(v)\\ U\neq \emptyset}} 
		\hspace{-0.5em}(-1)^{|U|+1} x(1+x)^{|\cap_{u \in U}N_{G-v}(u)|}
		+ \hspace{-0.5em}
		\sum_ {\substack{U\subseteq N_{G}(v)\\ U\neq \emptyset}} 
		\hspace{-0.5em}A_{U},	
	\end{equation*}
	where for $U\subseteq N_{G}(v)$ \\
	\begin{equation*}
		A_{U} = 
			\left\lbrace
				\begin{array}{l}
					x^{|U|} \text{ if } \bigcap_{u\in U}N_{G-v}(u)=\emptyset, \\
					0 \text{ otherwise.}		
				\end{array}
			\right.
	\end{equation*}
\end{corollary}

\section{The Neighborhood Polynomial of $k$-degenerate Graphs}\label{sect:degenerate}

Corollary \ref{coro:vertex-removal} suggests a recursion for the neighborhood 
polynomial of a graph based on removing the vertices of the graph one by one. 
In this section, we introduce $k$-degenerate graphs in which the calculation of 
the neighborhood polynomial is efficiently possible using the mentioned recursion.

A simple undirected graph $G=(V,E)$ of order $n$ is called \emph{$k$-degenerate} 
if every non-empty subgraph of $G$, including $G$ itself, has at least one vertex 
of degree at most $k$ for $0<k<n$.
Suppose $G=(V,E)$ is a $k$-degenerate graph of order $n$. To apply 
Corollary \ref{coro:vertex-removal} we need to select a vertex in $G$ to remove. 
Since $G$ is $k$-degenerate, the existence of a vertex of degree at most $k$ is 
guaranteed. Suppose $v_{1}$ is a vertex of degree $k$ in $G$, then we remove it 
and by using Corollary \ref{coro:vertex-removal}, we have 
\begin{align*}
	N(G,x) &= N(G-v_{1},x) + \sum_ {\substack{U\subseteq N_{G}(v_{1})\\ 
	U\neq \emptyset}} (-1)^{|U|+1} x(1+x)^{|\cap_{u \in U}N_{G-v_{1}}(u)|}\\
	&+\sum_ {\substack{U\subseteq N_{G}(v_{1})\\ U\neq \emptyset}} A_{U}.
\end{align*}

We define $G_{n}:=G$ where $n$ is the number of vertices of $G$. For each $i$, 
$0<i<n$, let $G_{n-i}$ be the graph obtained from 
$G_{n-(i-1)}$ after removing a vertex $v_{i}$ of degree at most $k$.  
The existence of such vertex is guaranteed by $k$-degeneracy of graph $G$.
We define for every non-empty $U\subseteq N_{G_{n-(i-1)}}(v_{i})$
\begin{equation*}
A_{U} = 
\left\lbrace
\begin{array}{l}
x^{|U|} \text{ if } \bigcap_{u\in U}N_{G_{n-i}}(u)=\emptyset, \\
0 \text{ otherwise.}		
\end{array}
\right.
\end{equation*} 
and for each $i$, $0<i<n$
\begin{equation*}
X_{i}=\sum _{U\subseteq N_{G_{n-(i-1)}}(v_{i}), U\neq \emptyset} [(-1)^{|U|+1}x(1+x)^{|\cap_{u\in U} N_{G_{n-i}}(u)|}+A_{U}]   
\end{equation*}
Then we have
\begin{align*}
	N(G_{n},x) &= N(G_{n-1},x)+X_{1} \\
	&= (N(G_{n-2},x)+X_{2})+X_{1} \\
	&\quad\vdots\\
	&=(\cdots((N(G_{1},x)+X_{n-1})+X_{n-2})+ \cdots )+X_{1})\\
	&= 1 + \sum_{i=1}^{n-1} X_{i}
\end{align*}
where the last equation is a result of the fact that $G_{1}$ is nothing than a 
single vertex which has neighborhood polynomial $1$.

Clearly, in each step (in other words for each $i$, $0<i<n$), we remove a vertex 
$v_{i}$ of degree at most $k$ (where its existence is guaranteed due to 
$k$-degeneracy of $G$) and calculate the term $X_{i}$ that is a sum over all 
non-empty subsets of $N_{G_{n-(i-1)}}(v_{i})$, but since this set has at most 
$k$ elements due to the degree of $v_{i}$ in $G_{n-(i-1)}$ the calculation of 
$N(G,x)$ can be done in polynomial time for any fixed $k$ which yields the following theorem.
\begin{theorem}\label{theo:k-degenerate}
	Let $k$ be a fixed positive integer. The calculation of the neighborhood 
	polynomial can be performed in polynomial-time in the class of $k$-degenerate
	graphs.
\end{theorem}
The recursive procedure for the calculation of the neighborhood polynomial becomes for graphs with
regular structure especially simple. We give here an example of a $2\times n$-grid graph that can
easily be generalized to similar graphs with regular structure.

\begin{figure}[ht]
	\centering
	\includegraphics[scale=0.9]{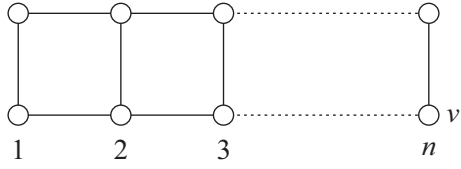}
	\caption{A ladder graph $L_n$}
	\label{fig:ladder}
\end{figure}

\begin{figure}[ht]
	\centering
	\includegraphics[scale=0.9]{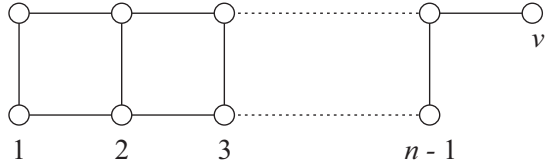}
	\caption{A modified ladder graph $M_n$}
	\label{fig:ladder1}
\end{figure}

\begin{example}
A \emph{ladder graph} of order $2n$, denoted by $L_n$, is the Cartesian product of two 
path graphs $P_{2}$ and $P_{n}$, see Figure \ref{fig:ladder}. The neighborhood polynomial of $L_n$ is
\begin{equation*}
    N(L_n,x) = 1 + 4x +2x^2 +(2n-4)x(1+x)^2 \quad \text{ for }n\geq 4.
\end{equation*}
To prove this equation we can calculate the neighborhood polynomial of $L_n$ by applying 
Theorem \ref{theo:cartesian}, but since $L_n$ is $2$-degenerate, we can also apply Corollary \ref{coro:vertex-removal}. 
To do so, we need to specify a vertex to remove. In the following, we remove a vertex of 
degree $2$, denoted by $v$ in Figure \ref{fig:ladder}. After removing the vertex $v$, 
we obtain a modified ladder graph that we denote by $M_n$, see Figure 
\ref{fig:ladder1}. The graph $M_n$ has exactly one vertex of degree $1$. This is the vertex 
to be removed in the next step; it is denoted by $v$ in Figure \ref{fig:ladder1}. 
 
By Corollary \ref{coro:vertex-removal}, we obtain
\begin{equation}
 	N(L_n,x) = N(M_n,x) + x(1+x)^2 \quad \text{ for }n\geq 4
 	\label{eq:L}
\end{equation}
and
\begin{equation}
 	N(M_n,x) = N(L_{n-1},x) + x(1+x)^2 \quad \text{ for }n\geq 5.
 	\label{eq:M}
\end{equation}
Substituting $N(M_n,x)$ in Equation (\ref{eq:L}) according to Equation (\ref{eq:M}) yields
\begin{equation}
 	N(L_n,x) = N(L_{n-1},x) + 2x(1+x)^2 \quad \text{ for }n\geq 4,
\end{equation}
which provides together with the initial value
\[ N(L_2,x) = N(C_4,x) =1+4x+2x^2 \] 
the above given result.
\end{example}

As any $k$-regular graph is $k$-degenerate, we obtain the following result.
\begin{corollary}
	Let $k$ be a fixed positive integer. The calculation of the neighborhood 
	polynomial can be performed in polynomial-time in the class of $k$-regular graphs.
\end{corollary}

As a consequence of Euler's polyhedron formula, each simple planar graph contains
a vertex of degree at most 5. This implies that any simple planar graph is 
5-degenerate, which provides the next statement.
\begin{corollary}
	The neighborhood polynomial of a simple planar graph can be found in polynomial-time.
\end{corollary}
There is an interesting generalization of planar graphs which was introduced in \cite{Gubser1996} that belongs 
to the class of $k$-degenerate graphs, too. An \emph{almost planar graph} is a non-planar graph $G=(V,E)$ 
in which for every edge $e\in E$, at least one of the graphs $G-e$ (obtained from $G$ after removing $e$) 
and $G/e$ (obtained from $G$ by contraction of $e$) is planar. It can be easily shown that every finite 
almost-planar graph is 6-degenerate, which implies that its neighborhood polynomial can be efficiently 
calculated.

\section{Conclusions and Open Problems}

The presented decomposition and reduction methods for the calculation of the neighborhood polynomial work
well for graphs of bounded degree or, more generally, for $k$-degenerate graphs. They can easily be extended
to derive the neighborhood polynomials of further graphs with regular structure such as grid graphs with
additional diagonal edges. The splitting formula for vertex separators also suggests that the neigborhood 
polynomial should be polynomial-time computable in the class of graphs of bounded treewidth.

A main open problem remains the calculation of the neighborhood polynomial of graphs for which the
number of edges is not linearly bounded by its order. Can we find an efficient way to calculate the 
neighborhood polynomial of graphs of bounded clique-width?

\printbibliography

\end{document}